\documentclass[12pt]{amsart}
\usepackage{amssymb}

\input amssym.def

 2   
\font\germ=eufm10 scaled 1200

\newtheorem{thm}{Theorem}[section]
\newtheorem{lem}[thm]{Lemma}
\newtheorem{prop}[thm]{Proposition}

\newtheorem{rmk}{Remark}[section] 
\newcommand{\thmref}[1]{Theorem~\ref{#1}}
\newcommand{\lemref}[1]{Lemma~\ref{#1}}

\newcommand{\propref}[1]{Proposition~\ref{#1}}

\newcommand{\Z}{{\mathbb Z}}
\newcommand{\Q}{{\mathbb Q}}
\newcommand{\N}{{\mathbb N}}
\newcommand{\F}{{\mathbb F}}
\newcommand{\A}{{\mathcal A}}
\newcommand{\M}{{\mathcal M}}
\newcommand{\p}{{\frak p}}
\newcommand{\half}{{\frac{1}{2}}}
\newcommand{\Gal}{{\rm Gal}}
\newcommand{\GL}{{\rm GL}}

\begin{document}

\title[Modular Forms]{A variant of Lehmer's conjecture, \linebreak
                                 II: The CM-case}

\author{Sanoli Gun and V. Kumar Murty}

\address[Sanoli Gun]
        {The Institute of Mathematical Sciences, 
         CIT Campus, Taramani, 
         Chennai 600 113 
         India}
\address[V. Kumar Murty]
      {Department of Mathematics,
       University of Toronto,
       40 St. George Street, 
       Toronto, ON, Canada, M5S 2E4.}
\email[Sanoli Gun]{sanoli@imsc.res.in }
\email[Kumar Murty]{murty@math.toronto.edu}

\maketitle

\begin{abstract}
Let $f$ be a normalized Hecke eigenform
with rational integer Fourier coefficients. 
It is an interesting question to know how often an
integer $n$ has a factor common with 
the $n$-th Fourier coefficient of $f$.
The second author \cite{kumar3} showed that this happens very often.
In this paper, we give an asymptotic formula for 
the number of integers $n$ for which $(n, a(n))=1$,
where $a(n)$ is the $n$-th Fourier coefficient of a
normalized Hecke eigenform $f$ of weight $2$ with rational integer
Fourier coefficients and has complex multiplication.
\end{abstract}

\section{Introduction}

\smallskip

The arithmetic of the Fourier coefficients of modular forms is intriguing and
mysterious. For instance, consider the cusp form of Ramanujan
\begin{eqnarray*}
\Delta(z) ~=~ \sum_{n=1}^{\infty} \tau(n) e^{2\pi i n z}.
\end{eqnarray*}
The coefficients $\tau(n)$ have received extensive arithmetic scrutiny
following the ground-breaking investigations of Ramanujan \cite{rama} himself.
Here, we have one of the oft-quoted conjectures in number theory attributed to 
Lehmer \cite{lehmer1},\cite{lehmer2} which asserts that 
$$
\tau(p) \ne 0, 
$$
where $p$ is a prime. Equivalently, for any $n \ge 1$,
$$
\tau(n) \ne 0.
$$
In general, proving such non-vanishing of all Fourier coefficients of 
a modular form is delicate and difficult. A more accessible problem is to study 
the arithmetic density of the non-zero coefficients. We refer to 
\cite{kumar1}, \cite{serre3} for results of this type.\\

In a recent work \cite{kumar3}, a variant of Lehmer's conjecture has been 
considered. More precisely, let
$$
f(z) ~=~ \sum_{n=1}^{\infty} a(n) e^{2\pi i n z}
$$
be the Fourier expansion of a normalized eigenform and 
suppose that the $a(n)$'s are rational integers for all $n$. Then it is natural 
to ask whether
$$
\# \left\{p \le x \mid a(p) \equiv 0 \!\!\pmod{p} \right\} ~=~ o(\pi(x)) ~.
$$
Heuristically, if the weight is $>2$, the number of such 
primes up to $x$ may grow like $\log\log x$
though we do not even know if these are of density zero. In general,
 denoting $(a,b)$ to be the greatest common divisor 
of $a$ and $b$, one can ask whether
$$
\# \left\{n \le x \mid (n, a(n)) \ne 1 \right\} ~=~ o(x),
$$
an assertion which turns out to be false. As mentioned in \cite{kumar3}, the 
correct
question in this context is the opposite assertion, namely whether it is true that
$$
\# \left\{n \le x \mid (n, a(n)) ~=~1 \right\} ~=~ o(x).
$$
This variant of Lehmer's conjecture appears to be amenable to study. In contrast
to the prime case, $a(n)$ almost always has a factor in common with $n$.
In particular, the following result has been proved in \cite{kumar3}. \\
Let us set 
$$
L_2(x) = \log\log x
$$ 
and for each $i\ge 3$, define 
$$
L_i(x) = \log L_{i-1}(x).
$$
In any occurence of an $L_i(x)$, we always assume that $x$ is sufficiently 
large so that $L_i(x)$ is defined and positive.

\begin{thm}\cite{kumar3}
For a normalized eigenform $f$ as above with rational integer Fourier 
coefficients $a(n)$, one has
$$
\# \left\{n \le x \mid (n, a(n)) ~=~1 \right\} ~\ll ~\frac{x}{L_3(x)}.
$$
\end{thm}

In the same paper, it was anticipated that if $f$ has 
complex multiplication (CM), a stronger result should hold. The 
ethos of our present work is to vindicate this anticipation, at least 
in the case that $f$ has weight $2$. 
A modular form $f$ is said to have CM if there is an imaginary
quadratic field $K$ and a Hecke character $\Psi$ of $K$ with conductor 
$\mbox{\germ m}$ so that
$$
f(z) ~=~ \sum_{\mbox{\germ a} \atop ~~~~\left(\mbox{\germ a}, \mbox{\germ m}\right)= 1} 
\Psi(\mbox{\germ a}) e^{2\pi i N(\mbox{\germ a})z}.
$$
Here, the sum is over integral ideals $\mbox{\germ a}$ of the ring of integers 
of $K$ which are coprime to $\mbox{\germ m}$ and $N(\mbox {\germ a})$ denotes 
the norm of $\mbox{\germ a}$. Thus 
$$
a(n) ~=~ \sum_{N(\mbox {\germ a})=n, \atop \left(\mbox{\germ a}, 
\mbox{\germ m}\right)=1} \Psi(\mbox {\germ a}).
$$
In particular for a prime $p$, $a(p) = 0$ if $p$ does not split in $K$ and $a(n)=0$
if $p||n$ (i.e. $p\mid n$ but $p^2\nmid n$) 
for some prime $p$ for which
$a(p) = 0$. It is well-known that if we are given a set $S$ of primes of
 positive density, the set of integers $n$ with the property 
that $p||n$ for some $p \in S$ has density one. Thus $a(n)=0$ for a set of $n$
of density one. More precisely, let us set
$$
M_{f,1}(x) = \# \left\{ n \le x \mid a(n) \ne 0 \right\}.
$$
Then we show that there is a constant $u_f$ so that
$$
M_{f,1}(x) ~=~ (1 ~+~ o(1)) \frac{u_f x}{\sqrt{\pi} (\log x)^\half}.
$$ 
We also show that there is a constant $\omega_f > 0$ so that
$$
\prod_{p < x \atop a(p) \neq 0}
\left( 1 - \frac{1}{p} \right)\ \sim\ \frac{\omega_f}{(\log x)^\half},
$$
where $\omega_f = \mu_f \mu_2\mu_3$, 
\begin{eqnarray*}
\mu_2 &=& \left\{ \begin{array}{ll}
                  \frac{1}{2} & \mbox{if $a(2)\ne 0$} \\
                    1  & \mbox{otherwise}
                  \end{array} \right.  \\ \\
\mu_3 &=& \left\{ \begin{array}{ll}
                  \frac{2}{3} & \mbox{if $a(3)\ne 0$} \\
                    1  & \mbox{otherwise}
                  \end{array} \right. 
\end{eqnarray*}
and $\mu_f$ is given in \propref{mertens}.
Finally, the main result of our paper is the following theorem.
\begin{thm}\label{thm1}
Let $f$ be a normalized eigenform of weight $2$ with rational integer 
Fourier coefficients $\{a(n)\}$. If $f$ is of CM-type, then 
there is a constant $U_f > 0$ so that 
\begin{eqnarray*}
\# \{n\le x \mid (n, a(n)) = 1\} = (1 + o(1))
                \frac{U_f x}{\sqrt{\pi} 
\left(L_3(x)\log x \right)^{\frac{1}{2}}}.
\end{eqnarray*} 
\end{thm}
The constant is given explicitly in terms of $f$ during the course 
of the proof.

\noindent
Our methods are based on the techniques of Erd{\"o}s \cite{erdos}, 
Serre \cite{serre1}, \cite{serre2} and those of
Ram Murty and the second author \cite{kumar4},
\cite{kumar2}, \cite{kumar3}, \cite{ram1},\cite{ram2}. 
Throughout this article, $p$ and $q$ will denote primes.

\section{Divisibility of fourier coefficients}\label{div}

Let $f$ be a normalized Hecke eigenform of weight $2$ for 
$\Gamma_0(N)$ with CM and let $K$ be the imaginary 
quadratic field associated to $f$. The Fourier
 expansion of $f$
at infinity is given by
$$
f(z) ~=~ \sum_{n=1}^{\infty} a(n)e^{2\pi i n z},
$$
where we are assuming that the $a(n)$'s are rational integers. 

\smallskip
For any prime $p$, let ${\Z}_p$ denote the ring of $p$-adic integers. By 
Eichler-Shimura-Deligne and since the Fourier coefficients of $f$ are in $\Z$, there 
is a continuous representation
$$
\rho_{p,f} ~:~ {\rm{Gal}}\left(\Bar{\Q}/{\Q}\right) \to GL_2({\Z}_p) .
$$
This representation is unramified outside the primes dividing $Np$. 
This means that
for any prime $q$ which does not divide $Np$ and for any prime $\mbox{\germ q}$
of $\Bar{\Q}$ over $q$, $\rho_{p,f}($Frob$_{\mbox{\germ q}})$ 
makes sense. We note that while
$\rho_{p,f}($Frob$_{\mbox{\germ q}})$ does depend on the choice of 
$\mbox{\germ q}$ over $q$, its
characteristic polynomial depends only on the conjugacy class of 
$\rho_{p,f}($Frob$_{\mbox{\germ q}})$(hence only on $q$) and is given by
\begin{eqnarray}\label{T}
T^2 - a(q)T + q.
\end{eqnarray}
We consider the reduction of the above representation modulo $p$
$$
\Bar{\rho}_{p,f} ~:~ {\rm{Gal}}\left(\Bar{\Q}/{\Q}\right) \to GL_2({\F}_p).
$$
The fixed field of the kernel of this representation determines a number field
$L$ which is a Galois extension of $\Q$ with group the image of 
$\Bar{\rho}_{p,f} $. 

\smallskip
We need to enumerate primes $q$ as above 
for which $a(q)\equiv 0\!\!\pmod{p}$. For this purpose, the following version 
of a theorem of Schaal \cite{schaal} is useful.

\begin{thm}\label{thmschaal}
Let $\mbox{\germ f}$ be an integral ideal of a number field $K$ of 
degree $n=r_1 + 2r_2$, where $r_1, r_2$ denote the number of real and complex 
embeddings respectively. Also let $\beta \in K$ denote an integer with 
$(\beta, {\mbox{\germ f}})=1$. Let $M_1, \cdots, M_{r_1}$ be nonnegative and
$P_1, \cdots, P_n$ be positive real numbers with $P_l = P_{l+ r_2}$,
$l= r_1+1, \cdots, r_1 + r_2$ and $P= P_1\cdots P_n$. Consider the number $B$ 
of integers $\omega \in K$ subject to the conditions:
 $$
 \omega \equiv  \beta \!\!\pmod{\mbox{\germ f}}, ~~~~~(\omega)~
{\text a~ prime~ ideal~} 
 $$
 $$  
   M_l \le  \omega^{(l)} \le M_l + P_l, ~~~~~l ~=~ 1, \cdots, r_1
  $$
  for real conjugates of $\omega$ and for complex conjugates
 $$
   |{\omega}^{(l)}| \le P_l, ~~~~~l= r_1+1, \cdots,n.
 $$
 If $P \ge 2$ and and the norm $\N{\mbox{\germ f}}$ satisfies
$$
\N{\mbox{\germ f}} \le \frac{P}{(\log P )^{(2r_1+ 2r_2 - 2 + 2/n)}}~,
$$
then one has \\
 $$
 B \ll  \frac{P}{\phi({\mbox{\germ f}}) \log{\frac{P}{\N{\mbox{\germ f}}}}}
\left\{1 + O\left(\log{\frac{P}{\N{\mbox{\germ f}}}}\right)^{-1/n}\right\},
 $$
where the implied constants depend only on $K$ and not on $\mbox{\germ f}$.
\end{thm}

 Define 
$$
\pi^{*}(x,p) := \# \left\{q\le x \mid a(q) \equiv 0 \!\!\!\!\!\pmod{p}, 
~a(q) \ne {0}\right\}.
$$

Now suppose that $q$ is a prime which splits in $K$, say 
$q {\mathcal O}_K=\mbox{\germ q}_1 \mbox{\germ q}_2$ and that
$\pi_q, {\Bar\pi}_q$ are roots of the characteristic polynomial \eqref{T}. 
Then
\begin{eqnarray*}
a(q) ~=~ \pi_q + {\Bar\pi}_q  ~~~~~~~~{\rm{and}}~~~~q~=~{\pi}_q {\Bar\pi}_q.
\end{eqnarray*}
Also if $a(q) \ne 0$, then ${\pi}_q \in {\mathcal O}_K$ and $|{\pi}_q| = q^{1/2}$.
If $a(q)\equiv 0\!\!\pmod{p}$, then ${\pi} _q^2 \equiv - q\!\!\pmod{p}$.
Thus, if in addition
$q \equiv a \!\!\pmod{p}$, then ${\pi}_q \!\!\pmod{p}$
has a bounded number of possibilities (at most $4$ in fact). Also, 
the ideal $(\pi_q)$ is prime as $(\pi_q)({\Bar\pi}_q ) = (q)$. Thus,
\begin{eqnarray*}
\sum_{q \le x \atop {\pi_q \equiv \alpha \!\!\!\!\!\pmod{p} 
\atop q \equiv a \!\!\!\!\pmod{p},
~q {\mathcal O}_K =\mbox{\germ q}_1\mbox{\germ q}_2}} 1
~~\le \sum_{\omega \in {\mathcal O}_K \atop {(\omega)~{\text is~prime},
~|\omega| \le \sqrt{x}
\atop \omega \equiv \alpha \!\!\! \pmod{p}}} 1.
\end{eqnarray*}
Applying \thmref{thmschaal} with $\mbox{\germ f} = (p)$, the right hand 
side is seen to be
$$
\ll \frac{x}{p^2 \log {\frac{x}{p^2}}}
$$ 
for $p^2 \le x/\log x$. \\

Now summing over all $a\!\!\pmod{p}$ yields the following proposition.

\begin{prop}\label{prop99}
Let $f$ be a modular form as above. Then for $p^2 \le x/\log{x}$, we have
$$
\pi^{*}(x,p) ~\ll~ \frac{x}{p \log {\frac{x}{p^2}}}.
$$
\end{prop}

Now using \propref{prop99} and partial summation, we see that for primes 
$p \le \sqrt{x/\log x}$,

\begin{eqnarray*}
\sideset{}{^*}{\sum}_{\substack{p^2 \log p \le q \le x \atop a(q)
\equiv 0 \!\!\! \pmod{p}}} \frac{1}{q}  ~\ll~ \frac{1}{p} 
\int_{p^2 \log p}^{x} \frac{dt}{t \log{\frac{t}{p^2}}}
~\ll~ \frac{1}{p} \log\log {\frac{x}{p^2}}~,
\end{eqnarray*}
where $\sideset{}{^*}{\sum}_{\substack{y \le q\le x}}$ means that the summation is over all 
primes $y \le q \le x$ for which $a(q) \ne {0}$. 
Thus, we have the following result.
\begin{prop}\label{prop100}
Let $f$ be a modular form as above and also let $p^2 \le {x}/{\log x}$ 
be a fixed prime. Then one has
\begin{eqnarray*}
\sideset{}{^*}{\sum}_{\substack{p^2 \log p \le q \le x  
\atop a(q) \equiv 0 \!\!\pmod{p}}} \frac{1}{q} ~\ll~ \frac{1}{p}
{L_2\left(\frac{x}{p}\right)},
\end{eqnarray*}
where $\sideset{}{^*}{\sum}_{\substack{y \le q\le x}}$ means that the summation is over all 
primes $y \le q \le x$ for which $a(q) \ne {0}$. 
\end{prop}

\begin{rmk}\label{remark 1}
We note that the contribution from the remaining primes \linebreak
$q \le p^2 \log p$ is
\begin{eqnarray*}
\sideset{}{^*}{\sum}_{\substack{q \le p^2 \log p \atop a(q) \equiv 0 \!\!\pmod{p}}}
 \frac{1}{q} ~\ll~ \frac{L_2(p)}{\log p}.
\end{eqnarray*}
However, we shall not make use of this estimate.
\end{rmk}

\section{Vanishing of $a(p)$}

Let $E$ be the elliptic curve defined over $\Q$ corresponding to the 
modular form $f$ of level $N = N_E$. 
As $f$ is of $CM$-type corresponding to the imaginary
quadratic field $K$, we know that $E$ has $CM$ by an order in $K$. 
A prime $p$ is supersingular for $E$ if $E$ has good
reduction at $p$ and its reduction $E_p$ has multiplication by an order
in a quaternion division
algebra. It is well known that a prime $p$ of good reduction
is supersingular if and only if 
\begin{equation}\label{supersingular2}
|E(\F_p)| \equiv 1 \pmod{p}.
\end{equation}
In particular, the set of primes supersingular for $E$ only depends on 
the isogeny class of $E$. For $p \geq 5$, \eqref{supersingular2} 
is equivalent to the condition $a(p) \ =\ 0$. 

\smallskip

Let ${\pi}_E(x)$ denote the  number of primes $p \le x$ such that $p$ is a 
supersingular prime for $E$. We know that
$$
\pi_E(x) \ge \pi_K^{-}(x),
$$ 
where $\pi_K^{-}(x)$ denotes the number of primes $p\le x$ that remain prime 
in $K$. In fact, the following more precise result is due to Deuring
(see \cite{Lang}, Chapter 13, Theorem 12).
\begin{prop}\label{Deuring}
Let $E$ be an elliptic curve defined over $\Q$
with multiplication by an order in an imaginary quadratic field $K$. 
Let $p$ be a prime of good reduction for $E$. Then 
$p$ is supersingular for $E$ if and only if $p$ ramifies or remains prime
in $K$. 
\end{prop}
In particular, this implies the following result.
\begin{prop}\label{supersingular}
Suppose that $p \geq 5$. 
With $E$ as in the previous proposition, we have $a(p)=0$ if and only if
$p$ is a prime of bad reduction or $p$ doesn't split in $K$. 
\end{prop} 
As $E$ has complex multiplication, it has additive reduction at 
primes of bad reduction and thus $a(p)=0$. The rest follows from Deuring's
result. 
\medskip\par
Finally, we record the following result which will be useful in
establishing the main result.
\begin{prop}\label{mertens}
There is a constant $\mu_f > 0$ so that
$$
\prod_{5 \leq p<z \atop a(p) \neq 0} \left( 1 - \frac{1}{p} \right)
\ =\ \frac{\mu_f}{(\log z)^\half}\ +\ O_f\left(\frac{1}{(\log z)^{3/2}}\right).
$$
\end{prop}
\begin{proof}
Using Rosen \cite{Rosen}, Theorem 2, we have
$$
\prod_{\N\p \leq z} \left( 1 - \frac{1}{\N\p} \right)^{-1}
\ =\ e^\gamma\alpha_K\log z\ +\ O_K(1).
$$
Here, the product is over primes $\p$ of $K$ and
$ \alpha_K $  is the residue at $s=1$ of the Dedekind zeta
function $\zeta_K(s)$. 
Note that $\alpha_K = L(1,\chi_K)$ where $\chi_K$ is the quadratic
character defining $K$ and $L(s,\chi_K)$ is the associated $L$-function.
It follows that
$$
\prod_{\N\p \leq z} \left( 1 - \frac{1}{\N\p} \right)
\ =\ \frac{e^{-\gamma}L(1,\chi_K)^{-1}}{\log z}
\ +\ O_K\left(\frac{1}{(\log z)^2}\right). 
$$
Thus, 
$$
\prod_{p \leq z \atop p {\rm\ splits\ in\ } K}
\left( 1 - \frac{1}{p} \right)\ =
\ \frac{\beta_K}{(\log z)^\half}\ +
\ O_K\left(\frac{1}{(\log z)^{3/2}}\right)
$$
where
$$
\beta_K\ =\ e^{-\gamma/2}L(1,\chi_K)^{-1/2}\prod_{p {\rm\ inert}}
\left( 1-\frac{1}{p^2}\right)^{-\half}\prod_{p|d_K}
\left( 1 - \frac{1}{p} \right)^{-\half}.
$$
By Proposition 
\ref{supersingular}, for $p \geq 5$, 
we have $a(p) \neq 0$ if and only if $p$
is a prime of good reduction and splits in $K$. This proves the
result with
$$
\mu_f\ =\ \beta_K\prod_{p {\rm\ splits} \atop p|6N}
\left( 1 - \frac{1}{p} \right)^{-1}.
$$
\end{proof}

\section{The number of non-zero fourier coefficients}

We begin by considering a slightly more general setting as in 
Serre \cite[\S 6]{serre2} which parts of this
section follow closely.
Let $n \mapsto a(n)$ be a multiplicative function and define the 
multiplicative function 
\begin{eqnarray*}
& a^{0}(n) =  \left\{\begin{array}{ll}
                1    & ~~~~\mbox{if $a(n) \ne 0$},\\
                0    & ~~~~\mbox{if $a(n) = 0$}.
                      \end{array} \right.
\end{eqnarray*}
We want the asymptotic behaviour of 
\begin{eqnarray*}
M_{a,d}(x) := \# \left\{ n\le x \mid a(n) \ne 0, ~~d|n \right\}
 = \sum_{dn \le x} a^{o}(dn),
\end{eqnarray*}
for any positive integer $d$. 
\subsection{The case $d=1$}\label{disone}
Consider the Dirichlet series
$$
\phi(s)\ =\ \sum_n \frac{a^0(n)}{n^s}\ =\ \prod_{p}\phi_p(s)
$$
where
$$
\phi_p(s)\ =\ \sum_{m=0}^\infty a^0(p^m)p^{-ms}.
$$
Let
$$
P_a(x)\ =\ \#\{p \leq x ~|~ a(p)=0\}.
$$
Suppose we know that 
\begin{eqnarray}\label{500}
P_{a}(x) = \lambda \frac{x}{\log x} ~+~ O\left(\frac{x}{(\log x)^{1+
\delta}} \right)
\end{eqnarray}
for some $\delta > 0$ and $\lambda < 1$. Then
$$
\sum_{p\le x} a^{0}(p) = (1-\lambda)\frac{x}{\log x} 
~+~ O\left(\frac{x}{(\log x)^{1+\delta}}\right)
$$
and
$$
\sum_{p}\frac{a^{0}(p)}{p^s} = (1 - \lambda) 
\log\left(\frac{1}{s-1}\right) ~+~ \epsilon_1(s),
$$
where $\epsilon_1(s)$ is analytic in a 
neighbourhood of $s=1$. Moreover,
\begin{eqnarray*}
\log(\phi(s)) =  
\sum_{p} \log(\phi_p (s)) = \sum_p \frac{a^{0}(p)}{p^s} 
~+~ \epsilon_2(s),
\end{eqnarray*}
where $\epsilon_2(s)$ is also analytic in a neighbourhood of $s=1$. Thus,
\begin{eqnarray*}
\log(\phi(s)) = (1-\lambda) \log\left(\frac{1}{s-1}\right) 
~+~  \epsilon_3(s)
\end{eqnarray*}
and 
$$
\phi(s) = \frac{e^{\epsilon_3(s)}}{(s-1)^{1-\lambda}}.
$$
A set of primes $P$ is called ``frobenien'' (in the sense of 
Serre (\cite{serre1}, Th\'eor\`eme 3.4))
if there is a finite Galois extension $K/{\Q}$
and a conjugacy-stable subset $H \subseteq G = {\rm Gal}(K/\Q)$ such that for $p$
sufficiently large, $p \in P$ if and only if $\sigma_p(K/{\Q}) \subseteq H$.
Here $\sigma_p(K/{\Q})$ denotes the conjugacy class of Frobenius 
automorphism associated to $p$.
If the set of primes enumerated by $P_a$ is ``frobenien'', 
we have
\begin{eqnarray}\label{100}
M_{a,1}(x)\ =\ \frac{u_{a}~ x}{\Gamma(1-\lambda) (\log x)^{\lambda}}
\ +\ O\left(\frac{x}{(\log x)^{\lambda+1}}\right),
\end{eqnarray}
where $u_{a} = e^{\epsilon_3(1)}$. Moreover, in the case that 
$\lambda = 0$, if one has the additional hypothesis that
\begin{eqnarray}\label{101}
\sum_{a(p) = 0} \frac{1}{p} < \infty
\end{eqnarray}
then \cite[p. 167]{serre2} states that
\begin{eqnarray}\label{102}
u_{a} = \prod_{a(p)=0} \left(1 - \frac{1}{p}\right).
\end{eqnarray}
\begin{rmk}
If we do not assume that $P_a$ enumerates a ``frobenien'' set of
primes, we can still invoke a Tauberian theorem to get an asymptotic
formula
$$
M_{a,1}(x)\ \sim\ \frac{u_a x}{\Gamma(1-\lambda)(\log x)^{\lambda}}.
$$
\end{rmk}
\medskip

In the next two subsections, we consider those arithmetic functions 
for which $P_a$ is frobenien.

\subsection{Convolution with a secondary function}\label{convolution}$~~~$\\

Now consider another function $n \mapsto b(n)$ with
the following properties: 
\begin{enumerate}
\item There is an integer $d$ so that $b(n) \neq 0$
implies that all prime divisors of $n$ are prime divisors of $d$.

\item We have $|b(n)| \leq\ 4^{\nu(n)}$ where $\nu(n)$ is the 
number of distinct prime divisors of $n$.
\end{enumerate}
Let us set
\begin{equation}\label{xi}
\xi_d(s)\ =\ \sum_{n=1}^\infty \frac{b(n)}{n^s}.
\end{equation}
We see that
$$
\sum_{m \leq x} |b(m)| \ \leq\ \sum_{p|m \Rightarrow p|d}
4^{\nu(m)}(x/m)^{1/4}\ =
\ x^{1/4}\prod_{p|d}\left( 1 + \frac{4}{p^{1/4}-1}\right).
$$
We observe that 
$$
\prod_{p|d} \left( 1 + \frac{4}{p^{1/4}-1}\right)
\ \ll\ 2^{\nu(d)}
$$
and so
\begin{equation}\label{sumofb}
\sum_{m \leq x} |b(m)|\ \ll\ x^{1/4}2^{\nu(d)}.
\end{equation}
Moreover,  using \eqref{sumofb},
we have
\begin{equation}\label{tailofx2}
\sum_{z < m < 2z} \frac{|b(m)|}{m}\ \ll\ z^{-3/4}2^{\nu(d)}.
\end{equation}
Let $c = a^0*b$ be the Dirichlet convolution and
consider the function
$$
\psi(s)\ =\ \sum_n \frac{c(n)}{n^s}\ =\ \phi(s)\xi_d(s).
$$
Then, we have
$$
\sum_{n \leq x} c(n)\ =\ \sum_{m \leq x} b(m) \sum_{r \leq x/m} a^0(r).
$$
The contribution from terms with $\sqrt{x}\leq m \leq x$ is 
$$
\leq x \sum_{\sqrt{x}\leq m \leq x} \frac{|b(m)|}{m}.
$$
Decomposing the sum into dyadic intervals $U < m \leq 2U$ and 
using \eqref{tailofx2} shows that the 
summation is $O(x^{-3/8}2^{\nu(d)})$
and hence the whole expression is $O(x^{5/8}2^{\nu(d)})$.
Assuming that \eqref{100} holds (that is, that $P_a$
enumerates a ``Frobenien'' set of primes), we have
\begin{equation}\label{sumcn}
\sum_{n \leq x} c(n)\ =\ 
\sum_{m \leq \sqrt{x}} b(m)\left\{
\left(\frac{u_a}{\Gamma(1-\lambda)} + O\left(\frac{1}{\log x}\right)\right)
\frac{x}{m(\log x/m)^\lambda}\right\}\ +\ O(x^{5/8}2^{\nu(d)}).
\end{equation}
Note that
$$
\left( \log \frac{x}{m} \right)^{-\lambda}
\ =\ \left( \log x \right)^{-\lambda}\ +\ O((\log m)(\log x)^{-\lambda-1}).
$$
Using this and \eqref{tailofx2}, the right hand side of \eqref{sumcn} is equal to
$$
\left(\frac{u_a}{\Gamma(1-\lambda)} + 
O\left(\frac{1}{\log x}\right)\right)\frac{x}{(\log x)^\lambda}
\left( \xi_d(1)\ +\ O(x^{-3/8}(\log x)^{-1}2^{\nu(d)})\right)\ +
\ O(x^{5/8}2^{\nu(d)}).
$$
Summarizing this discussion, we have proved the following.

\begin{prop}\label{averageofc}
We have
$$
\sum_{n\leq x} c(n)\ =\ \frac{u_a\xi_d(1)}{\Gamma(1-\lambda)}
\frac{x}{(\log x)^\lambda}\ +
\ O\left(\frac{x2^{\nu(d)}}{(\log x)^{\lambda+1}}\right)
$$
uniformly in $d$. 
\end{prop}

\medskip
\subsection{The case of general $d$}$~~~~$\\

Consider the Dirichlet series 
\begin{equation}
\psi_d(s) = \sum_{n}\frac{a^{0}(dn)}{n^s}.
\end{equation}
We may write it as
\begin{equation}
\left(\sum_{n_1=1 \atop p|n_1 \Rightarrow p|d}^\infty 
\frac{a^{0}(dn_1)}{n_1^s}
\right)
\left(\sum_{n_2=1 \atop (n_2,d)=1}^\infty
\frac{a^{0}(n_2)}{n_2^s}
\right).
\end{equation}
Thus, we see that
\begin{equation}
\psi_d(s)\ =\ \phi(s)\xi_d(s)
\end{equation}
where as in Section \ref{disone}
$$
\phi(s)\ =\ \sum_{n_3=1}^\infty \frac{a^{0}(n_3)}{n_3^s}
$$
and
$$
\xi_d(s)\ =\ \left(\sum_{n_1=1 \atop p|n_1 \Rightarrow p|d}^\infty 
\frac{a^{0}(dn_1)}{n_1^s}
\right)
\left(\sum_{n_2=1 \atop p|n_2 \Rightarrow p|d}^\infty
\frac{a^{0}(n_2)}{n_2^s}
\right)^{-1}.
$$
We have a factorization
\begin{eqnarray*}
\xi_d(s) 
= \prod_{p|d} \xi_{p,d}(s),
\end{eqnarray*}
where 
\begin{equation}
\xi_{p,d}(s)\ = \left(\sum_{m=0}^\infty a^0(p^{m+{\rm ord}_p d})p^{-ms}\right)
\left(\sum_{m=0}^{\infty} a^{0}(p^m) p^{-ms}\right)^{-1}.
\end{equation}
We record the following estimate for later use.
\begin{lem}\label{obvious}
$$
\xi_{p,d}(1)\ =\ a^0(p^{{\rm ord}_p d})\ +\ O\left(\frac{1}{p}\right).
$$
\end{lem}
\noindent
We write
$$
\xi_d(s)\ =\ \sum_{n=1}^{\infty}\frac{b(n)}{n^s}
$$
and suppose that $\xi_d(s)$ (that is, the coefficients $\{ b(n) \}$)
satisfies the conditions of Section \ref{convolution}. 
Recall that
$$
M_{a,d}(x):= \# \left\{ n\le x \mid a(n)\ne 0, \ d|n \ \right\}.  
$$
We have
$$
M_{a,d}(x)\ =\ \sum_{dn \leq x} a^0(dn)
$$
and by Proposition \ref{averageofc}, we deduce the following.
\begin{prop}\label{mad}
If $\xi_d$ satisfies the hypotheses of Section \ref{convolution}, then
we have
$$
M_{a,d}(x)\ =\ \frac{u_a\xi_d(1)}{\Gamma(1-\lambda)}
\frac{x/d}{(\log x/d)^\lambda}\ +
\ O\left(
\frac{x2^{\nu(d)}}{d(\log x/d)^{\lambda+1}}
\right) 
$$
uniformly in $d$.
\end{prop}

\subsection{Application to modular forms}$~~~~$\\

Now let $f$ be a normalized Hecke eigenform of weight $k \ge 2$ and 
let $a(n)= a_f(n)$ denote the $n$-th Fourier coefficient of $f$. 
In this case, let us denote the constant $u_{a}$ of the previous 
paragraph by $u_{f}$, and the function $M_{a,d}$ by $M_{f,d}$. 
\medskip\par
In some cases, $u_{f}$ can be made explicit. 
If $f$ does not have CM and $d=1$, then 
condition \eqref{101} holds (see \cite{kumar4}) 
and so $u_{f}$ is given by \eqref{102}.  
We shall discuss the case that $f$ has CM. 
\medskip\par
\noindent
In this case the assumption \eqref{500} made on $P_{a}(x)$ is true 
with $\lambda = \frac{1}{2}$ and so
$$
M_{f,1}(x)\ \sim\ \frac{u_f x}{\sqrt{\pi}(\log x)^\half}.
$$
(Here, we have used the fact that $\Gamma(\half)=\sqrt{\pi}$.) If
we assume that $f$ is of weight $2$ and has integer Fourier coefficients,
then by Proposition 3.2, the ``frobenien'' condition is
satisfied apart from a finite set of primes. 
If we can show that the conditions of Section \ref{convolution}
are satisfied, then 
specializing Proposition \ref{mad} to this case, 
we can deduce the following. 
\begin{prop}\label{nonzero}
We have
\begin{eqnarray*}
M_{f,d}(x) = \# \left\{n\le x \mid a_f(n) \ne 0, ~~d|n \right\} 
~=~ 
\frac{u_f x \displaystyle \xi_d(1)}
{\sqrt{\pi}d ~(\log x/d)^\half}~
+ O\left( \frac{x2^{\nu(d)}}{d(\log x/d)^{3/2}}\right)
\end{eqnarray*}
where $u_f$ is a constant depending on $f$.
\end{prop}
\medskip\par
We begin with some preliminary results. Let us set $i_f(p)$ to be the 
least integer $i \geq 1$ for which $a_f(p^i) = 0$. If for a given $p$,
there is no such $i$, then let us set $i_f(p) = 0$. 
In particular, if $p$ divides the level $N$ of $f$, then
$i_f(p)=1$. 
\begin{lem}\label{cnta}
For $p \nmid N$, we have
\begin{enumerate}
\item $i_f(p) \in \{ 0,1,2,3,5 \}$.
\item If $i_f(p) > 0$, then $a_f(p^i) = 0$ for every $i> 0$ with
              $$ i+1 \equiv 0\!\!\! \pmod {i_f(p)+1}.$$ 
\item If $a_f(p^i)=0$ for some $i>0$, then $i+1 \equiv 0 \pmod {i_f(p)+1}$. 
\item For $p$ sufficiently large (depending on $f$), we have $i_f(p) \in
\{ 0,1 \}$. 
\end{enumerate}
\end{lem}
\begin{proof}
Let us suppose that $i_f(p) > 0$. Thus, $a_f(p^i) = 0$ for some 
$i \geq 1$. Let us write $\alpha_p$ and $\beta_p$ for the roots of
$X^2 - a_f(p)X + p$. Then, we have
\begin{equation}\label{recursion}
a_f(p^i)\ =\ \frac{\alpha_p^{i+1}-\beta_p^{i+1}}{\alpha_p-\beta_p}.
\end{equation}
Thus, $\alpha_p = \zeta\beta_p$ where $\zeta^{i+1}=1$. Since
$\zeta \in {\Q}(\alpha_p,\beta_p)\ =\ {\Q}(\alpha_p)$
and $[\Q(\alpha_p):\Q]=2$, we must have $\zeta^2 = 1$ or
$\zeta^4 = 1$ or $\zeta^6 = 1$. This means that one of
$\{ \zeta+1, \zeta^2+1, \zeta^2+\zeta+1, \zeta^2-\zeta+1\}$ is zero. 
This in turn means that one of $\{ a_f(p), a_f(p^3), a_f(p^2), a_f(p^5)\}$
is zero. This proves the first assertion. The second follows from 
\eqref{recursion}. For the third assertion, we note that $\alpha_p = 
\zeta\beta_p$ where $\zeta^{i+1}=1$. We also have $\zeta^{i_f(p)+1}=1$.
Hence, $\zeta^j=1$ where $i+1 \equiv j \pmod {i_f(p)+1}$. If $j>0$, 
then $a_f(p^{j-1})=0$. But $ 0 \leq j-1 < i_f(p)$, a contradiction 
unless $j=1$. But then $a_f(1)=0$ which is also a contradiction. Hence,
we must have $j=0$, proving the third assertion. The fourth assertion 
follows from \cite{ram2}, Lemma 2.5. 
\end{proof}
As before, let us set
$$
\phi_p(s)\ =\ \sum_{m=0}^\infty a^0(p^m)p^{-ms}.
$$
From the above lemma, we deduce the following.
\begin{lem}\label{phip}
We have for $p \nmid N$,
$$
\phi_p(s)\ =\left\{ \begin{array}{llll}
\left(1-\frac{1}{p^s}\right)^{-1} & \mbox{if $i_f(p)=0$,}\\
p^s\left( \frac{1}{p^s-1} - \frac{1}{p^{(i_f(p)+1)s}-1} \right)
& \mbox{if $i_f(p) > 0$}.\\
\end{array} \right.
$$
Note $\phi_p(s)=1$ for $p\mid N$.
\end{lem}
Next, we evaluate $\xi_d(1)$. We have the following.
\begin{prop}\label{xi}
Writing
$$
\xi_d(s)\ =\ \sum_{n=1}^\infty \frac{b(n)}{n^s}
$$
we have that
\item $b(n) = 0$ if $n$ is divisible by a prime that does not divide $d$ and, 
\item if $p|d$, we have $|b(p^m)|\ \leq 4$
for all $m$. 

In particular, the function $n \mapsto b(n)$ satisfies the conditions 
of Section \ref{convolution}. Moreover, 
we have for $p \nmid N$,
$$
\xi_{p,d}(1)\ = \left\{
\begin{array}{lll}
1 & \mbox{if $i_f(p)=0$,}\\
1+ p^{-1} - p^{v- 2k_0 + 1} & \mbox{if $i_f(p)=1$,}\\ 
\frac{1+p +\cdots + p^{i_f(p)}- p^{v - (k_0-1)(i_f(p)+1)}}{p+\cdots + p^{i_f(p)}}
& \mbox{if $i_f(p) > 1$.}\\
\end{array} \right.
$$
Here $v = {\rm ord}_p d$ and $k_0$ is the smallest integer $\ge \frac{v+1}{i_f(p) +1}$.
\end{prop}
\proof
By a calculation similar to that of Lemma \ref{phip}, we see that
$$
\sum_{m=0}^\infty a^0(p^{m+ v})p^{-ms}
\ =\left\{ \begin{array}{llll}
\left(1-\frac{1}{p^s}\right)^{-1} & \mbox{if $i_f(p)=0$,}\\
p^s\left( \frac{1}{p^s-1} - \frac{p^{\{v - (k_0-1)(i_f(p)+1)\}s}}
{p^{(i_f(p)+1)s}-1} \right)
& \mbox{if $i_f(p) > 0$.}\\
\end{array} \right.
$$
Hence, writing $i = i_f(p)$, we have
$$
\xi_{p,d}(s)\ =\ \frac{p^{(i+1)s} - 1 - p^{\{v+1-(k_0-1)(i+1)\}s} + p^{\{v- (k_0-1)(i+1)\}s}}
{p^{(i+1)s} - p^s}
$$
which is equal to
$$
\left( 1 - \frac{1}{p^{\{k_0(i+1)-v-1\}s}}
+ \frac{1}{p^{\{k_0(i+1)-v\}s}} - \frac{1}{p^{(i+1)s}} \right) 
\left( 1 - \frac{1}{p^{is}} \right)^{-1}
$$
from which it follows that $|b(p^m)| \leq 4$.
Moreover, as 
$$
\xi_d(s) = \prod_{p|d} \xi_{p,d}(s)
$$
it follows also that $b(n)=0$ unless every prime divisor of $n$
also divides $d$.  The last assertion of the lemma follows from the
above formulas.

\begin{rmk}
Note that the dependence on $d$ of $\xi_{p,d}$ is only through ${\rm ord}_p d$.
Thus, where the meaning is clear, for $p|d$ and $d$ squarefree, we shall write
$\xi_p$.
\end{rmk}

\medskip\par
In the remainder of this section, we will elaborate on the constant
$u_f$ and in particular, relate it to $L$-function values. 
From Lemma \ref{phip}, we have
$$
\log \phi(s) = - \sum_{i_f(p)=0}\log\left(1- \frac{1}{p^s}\right) 
- \sum_{i_f(p)=1} \log\left(1 - \frac{1}{p^{2s}}\right)
+ \sum_{i_f(p)>1} \log \phi_p(s).
$$
Note that by Lemma \ref{cnta}, (4), the third sum on the right hand side 
ranges over a finite set of primes $p$.
\medskip\par
Denote by $\chi_K$ the quadratic Dirichlet character 
that defines $K$ and $L(s, \chi_K)$ the associated Dirichlet series. 
Let us denote by $S, I, R$ the
set of primes that split, stay inert or ramify in $K$ (respectively). 
Then, we have
\begin{eqnarray*}
- \sum_{p\in S} \log\left(1- \frac{1}{p^s}\right) 
&=& \frac{1}{2}\log\zeta(s) + \frac{1}{2}\log L(s, \chi_K) +
 \frac{1}{2}\sum_{p\in I}
\log\left(1 - \frac{1}{p^{2s}}\right) \\
 && + ~\frac{1}{2}\sum_{p\in R} \log\left(1- \frac{1}{p^s}\right) \\
\end{eqnarray*}
Moreover,
if $i_f(p)=0$ then $a(p) \neq 0$ and for $p\nmid 6N$, this 
means that $p$ is a prime of good reduction and splits in $K$. Therefore, 
$$
- \sum_{i_f(p)=0 \atop p \nmid 6N} \log\left(1 - \frac{1}{p^s}\right) 
= - \sum_{p\in S \atop p \nmid 6N}
\log\left(1 - \frac{1}{p^s} \right) 
+ \sum_{i_f(p)>1 \atop p \nmid 6N} 
\log\left(1 - \frac{1}{p^s}\right).
$$
Since $i_f(p)=1 \Leftrightarrow a(p) = 0$, we can write
$$
-\sum_{i_f(p)=1 \atop p\nmid 6N} \log\left(1 - \frac{1}{p^{2s}}\right)
= -\sum_{a(p)=0 \atop p\nmid 6N} \log\left(1 - \frac{1}{p^{2s}}\right).
$$
After a straightforward (but tedious) computation, one sees that
\begin{eqnarray*}
\log \phi(s)&=& \frac{1}{2} \log\frac{1}{s-1} + \frac{1}{2} 
\log\left(\zeta(s)(s-1)\right) 
+ \frac{1}{2} \log L(s, \chi_K) \\
&+& 
\half \sum_{p \in I} \log \left( 1 - \frac{1}{p^{2s}} \right)
\ + \ \log C(s), 
\end{eqnarray*}
where
\begin{eqnarray*}
C(s) &=&
\prod_{a(p)=0 \atop p\nmid 6N} \left( 1 - \frac{1}{p^{2s}} \right)^{-1}
\prod_{p\in R} \left( 1 - \frac{1}{p^s} \right)^{\half}
\prod_{p \in S \atop p|6N} \left( 1 - \frac{1}{p^s} \right)\\
&&
\prod_{i_f(p)> 1 \atop p \nmid 6N} 
\left\{\left( 1 - \frac{1}{p^{s}} \right)\phi_p(s)\right\}
\prod_{p|6N} \phi_p(s).\\
\end{eqnarray*}
Putting the above discussion together, we see that
$$
\phi(s) = \frac{\epsilon(s)}{(s-1)^{1/2}}~,
$$
where 
\begin{eqnarray*}
u_{f} &=& \epsilon(1) \\
&=& L(1, \chi_K)^{1/2}\prod_{p \in I} 
\left(1 - \frac{1}{p^2} \right)^{1/2}C(1). \\
\end{eqnarray*}
\medskip\par

\section{A sieve lemma}

We record a simple consequence of \propref{nonzero} that will be 
used in section \ref{proofofmainresult}.
\begin{lem}\label{sieve1}
Let $f$ be as in the previous section, that is, a normalized Hecke
eigenform of weight $ \geq 2$ with complex multiplication.
Let $y_1 = L_2(x)^{1+\epsilon}$ and set
\begin{equation}\label{Ny1}
N_{y_1}(x)\ =\ \{ n \leq x: q|n 
\Rightarrow q \geq y_1, a_f(n) \neq 0\}.
\end{equation}
Then
\begin{equation}\label{Ny1estimate}
N_{y_1}(x)\ =\ \frac{U_fx}{\sqrt{\pi}(L_3(x)\log x)^{\half}}
\ +\ O\left(\frac{xL_3(x)^2}{(\log x)^{3/2}}\right),
\end{equation}
where
\begin{equation}
U_f\ =\ \frac{u_f\mu_fc_f}{\sqrt{\pi}}
\prod_{p < y_1 \atop i_f(p) > 1} \left( 1 - \frac{\xi_{p,d}(1)}{p} \right)
\prod_{p \in \{ 2,3\} \atop i_f(p) = 0} \left( 
1 - \frac{1}{p} \right).
\end{equation}
Note that the last two products are over a finite number of primes
and 
$$
c_f = \prod_{5 \le p < y_1 \atop i_f(p) \ge 2} \left( 1 - \frac{1}{p} \right)^{-1}
\prod_{p < y_1 \atop i_f(p)= 1} \left( 1 - \frac{1}{p^2} \right). 
$$

\medskip\par
\end{lem}
\begin{proof}
Set $P_{y_1} \ =\ \prod_{p < y_1} p$.
By the principle of inclusion-exclusion, we have
$$
N_{y_1}(x) = \displaystyle\sum_{d \mid P_{y_1}} \mu(d) M_{f,d}(x). 
$$
Since $P_{y_1} \ll e^{y_1}$, we see that for any $d|P_{y_1}$,
we have $\log x\ \ll\ \log x/d\ \ll\ \log x$. 
Now using \propref{nonzero}, the right hand side is
$$
= \frac{u_f x}{\sqrt{\pi}(\log x)^{\frac{1}{2}}} 
\displaystyle\sum_{d \mid P_{y_1}} 
\frac{\mu(d)}{d}
\left(\xi_d(1) + O\left(\frac{2^{\nu(d)}}{(\log x)}\right)\right).
$$
The main term is
\begin{eqnarray*}
&=& \frac{u_f x}{\sqrt{\pi}(\log x)^{\frac{1}{2}}} 
\displaystyle\prod_{p < y_1} \left(1-\frac{\xi_{p,d}(1)}{p}\right) \\
&=& \frac{u_fx}{\sqrt{\pi}
\left( \log x \right)^{\frac{1}{2}}}
\prod_{5 \leq p < y_1 \atop i_f(p)= 0} \left( 1 - \frac{1}{p} \right) 
\prod_{p < y_1 \atop i_f(p) \ge 1}\left( 1 - \frac{\xi_{p,d}(1)}{p} \right)
\prod_{p \in \{ 2,3\} \atop i_f(p)=0} \left(
1 - \frac{1}{p} \right).\\
&=& \frac{u_fx}{\sqrt{\pi}
\left( \log x \right)^{\frac{1}{2}}}
\prod_{5 \leq p < y_1 \atop i_f(p)= 0} \left( 1 - \frac{1}{p} \right) 
\prod_{p \le y_1 \atop i_f(p)= 1} \left( 1 - \frac{1}{p^2} \right)\\
&& \phantom{mmmmmmmm}
\prod_{p < y_1 \atop i_f(p)> 1}\left( 1 - \frac{\xi_{p,d}(1)}{p} \right)
\prod_{p \in \{ 2,3\} \atop i_f(p)=0} \left(
1 - \frac{1}{p} \right).
\end{eqnarray*}
Note that if $i_f(p)=1$ and $d$ is squarefree, we have by \propref{xi},
$\xi_{p,d}(1)= \frac{1}{p}$. Also note that by \lemref{cnta}, there are only 
finitely many primes $p$ for which $i_f(p)>1$, ensuring the convergence
of 
$$
\prod_{i_f(p)> 1}\left( 1 - \frac{\xi_{p,d}(1)}{p} \right).
$$ 
Now using \propref{mertens}, we see that the above sum is
$$
\frac{U_fx}{\sqrt{\pi}(L_3(x)\log x)^\half}.
$$
The error term is
$$
\ll \frac{x}{(\log x)^{3/2}}\sum_{d|P_{y_1}}
\frac{|\mu(d)|}{d}2^{\nu(d)}.
$$
The sum over $d$ is
\begin{eqnarray*}
&\ll\ & \prod_{\ell < y_1} \left( 1 + \frac{2}{\ell}\right) \\
& \ll\ &\prod_{\ell < y_1} \left( 1 - \frac{1}{\ell}\right)^{-2} \\
& \ll\ & L_3(x)^2.
\end{eqnarray*}
This proves the result.
\end{proof}
We record here a variant of the above result.
\begin{lem}\label{sieve2}
Suppose that $p \le y_1$. We have
\begin{eqnarray*}
 &&\#\{n\le x \mid p|n, a_f(n) \neq 0, q|n \Rightarrow q \geq p \} \\ 
 && \phantom{mmm} \ll  \frac{x}{p(\log x)^\half}
\prod_{\ell \le p \atop \ell \text{ prime}} 
\left( 1 - \frac{1}{\ell} \right) 
+ \frac{x}{(\log x)^{3/2}}e^{4\sqrt{p}}\frac{\log p}{p}.
\end{eqnarray*}
\end{lem}

\section{Siegel zeros}

Let $L/\Q$ be a Galois extension of number fields with group $G$ and
$n_L$, $d_L$ be the degree and the absolute value of the
discriminant of $L/{\Q}$ respectively. Suppose that Artin's conjecture
on the holomorphy of Artin $L$-functions is known for $L/\Q$. Set
$$
\log \M\ =\ 2\left(\sum_{p|d_L} \log p\ +\ \log n_L \right) .
$$
Also, denote by $d$ the maximum degree and by 
$\A$ the maximum Artin conductor of an irreducible character of $G$. 
\medskip\par
Let $C$ be the set of elements in $G$ that map to the Cartan subgroup
and also have trace zero. Then $C$ is stable under conjugation and
thus $C$ is a union of conjugacy classes.
Denote by $\pi(x,C)$ the number of primes $p\le x$ with ${\rm Frob}_p \in C$.
Then, \cite{kumar4}, Theorem 4.1 
asserts that for
$$
\log x \gg d^4 (\log \M),
$$
there is an absolute and effective constant $c > 0$ so that
$$
\pi(x,C) = \frac{|C|}{|G|}{\rm Li}\ x - \frac{|C|}{|G|}{\rm Li}\ x^{\beta}
 + O\left(|C|^\half x~(\log x\M)^2 
\exp{\left\{\frac{-c\log x}{d^{3/2}\sqrt{d^3(\log \A)^2+\log x}}
\right\}}\right).
$$
The term involving $\beta$ is present only if the Dedekind zeta function 
$\zeta_L(s)$ of $L$ has a real zero $\beta$ (the Siegel zero), in the interval
$$
1- \frac{1}{4 \log d_L} \le \Re(s) < 1.
$$
Let $L$ be the fixed field of the kernel of $\bar{\rho}_{p,f}$. 
(Recall that $\bar{\rho}_{p,f}$ was introduced in Section 
\ref{div}.)
Now, let $G = \Gal(L/\Q)$ (viewed as a subgroup of $\GL_2(\Z/p)$) 
and let $C$ be the subset of elements of $G$ of trace zero. It is known that
the subgroup $H = \Gal(L/K)$ is Abelian and maps to a Cartan subgroup
of $\GL_2(\Z/p)$. The image of $G$ maps to the normalizer of this
subgroup. As $G$ has an Abelian normal subgroup of index $2$, it is 
well-known that all irreducible characters of $G$ are monomial, and so
Artin's holomorphy conjecture holds for it. 
\medskip\par
Thus, we can appeal to the above version of the Chebotarev density theorem.
The extension $L/K$ is unramified outside of primes dividing $pN$ where
$N$ is the level of $f$. We have $d=2$,  and
$$
\log \M \ll\ \log pN
$$
as well as
$$
\log \A\ \ll\ \log pN.
$$
For $p$ sufficiently large, it is known that $G$ maps {\it onto}
the normalizer of a Cartan subgroup, and hence 
$$
p^2 \ \ll\ |G|\ \ll\ p^2.
$$
Moreover, the size of $|C|$ satisfies
$$
p \ \ll\ |C|\ \ll\ p.
$$
Thus, if we set $\delta(p) \ =\ |C|/|G|$, we have 
$$
\frac{1}{p}\ \ll\ \delta(p)\ \ll\ \frac{1}{p}
$$
for $p$ sufficiently large.
Thus, we have the following result.
\begin{thm}\label{thm 5.1}
Let $f$ be a CM form of level $N$ as before. Then for 
$\log x \gg (\log pN)^2$, we have
$$
\pi^{*}(x,p) = \delta(p){\rm Li}\ {x} 
- \delta(p){\rm Li}\ {x^\beta} + O(x e^{-c\sqrt{\log x}}),
$$
where $\frac{1}{p} \ll \delta(p) \ll \frac{1}{p}$ 
and the implied constant is absolute and effective.
\end{thm}
From the discussion above, we know 
that the stated bounds on $\delta(p)$ hold for $p$
sufficiently large.  To deduce that they hold for all $p$, 
it suffices to show 
that $\delta(p)>0$ holds for all $p$. This inequality 
follows from the fact that the image of 
complex conjugation is an element of trace zero in the Galois group.
\medskip\par
If the Dedekind zeta function $\zeta_L(s)= 0$ has a Siegel zero $\beta$
with $1- \frac{1}{4 \log d_L} \le \Re(s) < 1$, then by
a result of Stark \cite[p. 145]{S} we know that 
there is 
a quadratic field $M$ contained in $L$ such that $\zeta_M(\beta)=0$. 
Further \cite[p. 147]{S}, for such $M$
$$
\beta < 1 - \frac{1}{\sqrt{d_M}}.
$$
Let $[L:M] = n$. Since $d_L \ge d_M^{n}$, we have
$$
\beta < 1 - \frac{1}{d_L^{1/2n}}.
$$
Now by an inequality of Hensel \cite[p. 129]{serre2},
$$
\log d_L \le 2n \log pn_L
$$
and so 
$$
\frac{1}{2n} \log {d_L} \le \log pn_L.
$$
Hence
\begin{eqnarray}\label{stark}
\beta < 1 - \frac{1}{pn_L}.
\end{eqnarray}

\section{Intermediate results}

As before 
$$
\pi^{*}(x,p) = \# \left\{q\le x \mid a(q) \equiv 0 \!\!\!\!\!\pmod{p}, 
~a(q) \ne {0}\right\}.
$$
Proving \thmref{thm1} requires the following lemmas. Let
$0< \epsilon < 1/2$ and set $y ~=~ L_2^{1 - \epsilon}(x)$.

\begin{lem}\label{1}
Let $p < y$ be a fixed prime. Then we have
\begin{eqnarray*}
\sideset{}{^*} 
{\sum}_{\substack{q \le x  \atop a(q) \equiv 0 \!\!\pmod{p}}} 
\frac{1}{q} = 
 \delta(p)L_2(x) + O(L_3(x)),
\end{eqnarray*}
where ${\sum}_{q\le x}^{*}$ means that the summation is over all 
primes $q \le x$ for which $a(q) \ne {0}$.
\end{lem}

\begin{proof}
By partial summation, the sum is
\begin{eqnarray*}
\sideset{}{^*}{\sum}_{\substack{q \le x  
\atop a(q) \equiv 0 \!\!\pmod{p}}}\frac{1}{q} ~=~ 
\frac{\pi^{*}(x,p)}{x} ~+~ \int_{2}^{x} \frac{\pi^{*}(t,p)}{t^2}~dt.
\end{eqnarray*}
But $\int_{2}^{x} \frac{\pi^{*}(t,p)}{t^2} dt$ can be written as
\begin{eqnarray*}
\int_{2}^{(\log x)^{\gamma}} \frac{\pi^{*}(t,p)}{t^2} ~dt  
~+~ \int_{(\log x)^{\gamma}}^{x} \frac{\pi^{*}(t,p)}{t^2} ~dt,
\end{eqnarray*}
where $\gamma$ is chosen in such a way that for 
$(\log x)^{\gamma} \le t \le x$,
we have
$\log t \gg (\log p N)^2$. 
The first integral is 
\begin{eqnarray*}
\le \int_{2}^{(\log x)^{\gamma}} \frac{\pi (t)}{t^2} ~dt ~~ \ll L_3(x),
\phantom{m} {\rm where}~~\pi(t) = \# \{p\le t ~\mid p~{\rm prime} \}
\end{eqnarray*}
and the second integral is  
\begin{eqnarray*}
\int_{(\log x)^{\gamma}}^{x} \frac{1}{t^2} 
\left(\delta(p){\rm Li}(t) - \delta(p) {\rm Li}{(t^{\beta})} 
+ O(t e^{-c \sqrt{\log t}})\right) ~dt, \phantom{m} {\rm by~ \thmref{thm 5.1}}. 
\end{eqnarray*}
The first term  is equal to
\begin{eqnarray*}
\delta(p) \int_{(\log x)^{\gamma}}^{x} \frac{dt}{t \log t} + O (L_3(x))\\ 
= \delta(p) L_2(x) + O(L_3(x)).    
\end{eqnarray*}
Next, consider the term with the Siegel zero. Since by \eqref{stark}, 
$\beta <1 - \frac{1}{pn_L}$, therefore the second term is 
\begin{eqnarray*}
\delta(p) \int_{(\log x)^{\gamma}}^{x} \frac{1}{t^2} {\rm Li}(t^{\beta})~dt 
&=& \delta(p) \int_{(\log x)^{\gamma}}^{x}\frac{dt}{t^2}\int_{2}^{t^{\beta}}
\frac{du}{\log u}\\
&=& \delta(p) \int_{2}^{x^{\beta}} \frac{du}{\log u} 
\int_{{\rm max}\left((\log x)^{\gamma}, u^{\frac{1}{\beta}}\right)}^{x}
\frac{dt}{t^2}.  
\end{eqnarray*}
We split the range of integration of $u$ into two integrals:
\begin{eqnarray*}
&(I)&   2\le u \le (\log x)^{\gamma\beta},\\
&(II)&  (\log x)^{\gamma\beta} \le u \le x^{\beta}.
\end{eqnarray*}
The first range gives rise to the integral
$$
\delta(p) \int_{2}^{(\log x)^{\gamma\beta}} \frac{du}{\log u} 
\left\{\frac{1}{(\log x)^{\gamma}} - \frac{1}{x} \right\}
~\ll~ \delta(p) (\log x)^{\gamma(\beta -1)} ~\ll~ 1.
$$
The second range gives rise to the integral
$$
\delta(p) \int_{(\log x)^{\gamma\beta}}^{x^{\beta}} 
\frac{du}{\log u} \left\{\frac{1}{u^{\frac{1}{\beta}}} - \frac{1}{x}\right\}.
$$
Set $v = u^{\frac{1}{\beta}}$. Then $v^{\beta} = u$ and 
$\beta \log v = \log u$.
Moreover, $du = \beta v^{\beta -1}dv$. Hence the integral is
\begin{eqnarray*}
\delta(p) \int_{(\log x)^{\gamma}}^{x} \frac{\beta v^{\beta-1}dv}{\beta \log v} 
\left(\frac{1}{v} - \frac{1}{x}\right) 
&\ll & \frac{\delta(p)}{(\log x)^{\gamma(1-\beta)}} 
   \int_{(\log x)^{\gamma}}^{x} \frac{dv}{v \log v}\\
&\ll & \frac{\delta(p)L_2(x)}{(\log x)^{\frac{\gamma}{pn_L}}}\\
&\ll & \frac{\delta(p)L_2(x)}{e^{{\frac{\gamma}{n_L}}{L_2(x)^{\epsilon}}}} ~\ll~1.
\end{eqnarray*}

Finally, using the elementary estimate $e^{c \sqrt{u}}\gg u^2$,
we deduce that the O-term is
\begin{eqnarray*}
\ll \int_{L_2(x)}^{\log x} \frac{du}{u^2} \ll 1 .
\end{eqnarray*}
The term $\pi^{*}(x,p)/x$ is of smaller order. This proves the 
lemma.
\end{proof}
\medskip\par
Define $\nu(p,n) = \# \{ q^m||n \mid a(q^m) \equiv 0 \!\! \pmod{p}\}$.

\begin{lem}\label{2}
Assume that $p < y$. Then we have
\begin{eqnarray*}
\sideset{}{^*}{\sum}_{\substack{n\le x}} \nu(p,n) =  (1 + o(1))
{\frac{u_f \delta(p)x L_2(x)}{\sqrt{\pi\log x}}}
~+~ O\left(\frac {x L_3(x)}{\sqrt{\log x}}\right),
\end{eqnarray*}
where ${\sum}_{n\le x}^{*}$ means that the summation is over all 
natural numbers $n \le x$ such that $a(n) \ne 0$.
\end{lem}

\begin{proof}
Interchanging summation, we see that
\begin{eqnarray*}
\sideset{}{^*}{\sum}_{\substack{n\le x}} \nu(p,n) 
&=& \sideset{}{^*}{\sum}_{\substack{q^m \le x  \atop a(q^m) 
\equiv 0 \!\!\!\pmod{p}}} ~\sideset{}{^*}{\sum}_{\substack{n \le x  \atop q^m ||n}} 1 \\
&=& \sideset{}{^*}{\sum}_{\substack{q\le x  
\atop a(q) \equiv 0 \!\!\!\pmod{p}}} \sideset{}{^*}{\sum}_{\substack{n\le x  \atop q||n}} 1 
\phantom{m} + \phantom{m} 
\sideset{}{^*}{\sum}_{\substack{q^m \le x, m\ge 2  \atop a(q^m)
\equiv 0 \!\!\! \pmod{p}}} \sideset{}{^*}
{\sum}_{\substack{n\le x  \atop q^m ||n}} 1 .
\end{eqnarray*}
The contribution of terms $q^m$ with $m \ge 2$ is 
\begin{eqnarray*}
\sideset{}{^*}{\sum}_{\substack{q^m \le x, \atop {m\ge 2 
 \atop a(q^m) \equiv 0\! \!\!\!\pmod{p}}}}
\sideset{}{^*}{\sum}_{\substack{n\le x  \atop q^m ||n}} 1
&=& 
\sideset{}{^*}{\sum}_{\substack{q^m \le x^{\epsilon}
\atop {m\ge 2  \atop a(q^m) \equiv 0\! \!\!\!\pmod{p}}}} 
\sideset{}{^*}{\sum}_{\substack{n\le x  \atop q^m ||n}} 1
\phantom{m} + \phantom{m} 
\sideset{}{^*}{\sum}_{\substack{x^{\epsilon} \le q^m \le x 
\atop {m\ge 2  \atop a(q^m) \equiv 0\! \!\!\!\pmod{p}}}} 
\sideset{}{^*}{\sum}_{\substack{n\le x  \atop q^m ||n}} 1 \\ \\
&\ll & 
\sideset{}{^*}{\sum}_{\substack{q^m \le x^{\epsilon}
\atop {m\ge 2  \atop a(q^m) \equiv 0\! \!\!\!\pmod{p}}}} 
\sideset{}{^*}{\sum}_{\substack{n\le x/q^m }} 1
\phantom{m} + \phantom{m}  
x \sideset{}{^*}{\sum}_{\substack{x^{\epsilon} \le q^m \le x
\atop  m \ge 2}} \frac{1}{q^m} 
\end{eqnarray*}
\begin{eqnarray*}
&\ll & 
\frac{x}{(\log x)^{\frac{1}{2}}}
\sideset{}{^*}{\sum}_{\substack{q^m \le x^{\epsilon} 
\atop m \ge 2}} \frac{1}{q^m}
\phantom{m} + \phantom{m}  
x \int_{x^{\epsilon}}^{x} \frac{dt}{t^2}, 
~~~{\rm by~ \propref{nonzero}} \\ \\
&\ll & 
\frac{x}{\sqrt{\log x}} \phantom{m} + \phantom{m}
 \frac{x}{x^{\epsilon}} 
~~\ll ~~
\frac{x}{\sqrt{\log x}}.
\end{eqnarray*}
Also, we have
\begin{equation}\label{20}
\sideset{}{^*}{\sum}_{\substack{q\le x  
\atop a(q) \equiv 0 \!\!\!\pmod{p}}} \sideset{}{^*}{\sum}_{\substack{n\le x  \atop q||n}} 1 
=
\sideset{}{^*}{\sum}_{\substack{ q \le x^{1/\log\log x}  
\atop a(q) \equiv 0\!\!\! \pmod{p}}}
\sideset{}{^*}{\sum}_{\substack{n\le x  \atop q||n}} 1 
\phantom{m} + \phantom{m} 
\sideset{}{^*}{\sum}_{\substack{x^{1/\log\log x} \le q\le x} 
\atop a(q) \equiv 0\!\!\! \pmod{p}}
\sideset{}{^*}{\sum}_{\substack{n\le x  \atop q||n}} 1. 
\end{equation}
We show that the second double sum on the right of
\eqref{20} contributes a negligible
amount. Indeed, consider first the quantity
\begin{equation}\label{october1}
\sideset{}{^*}{\sum}_{\substack{x^\epsilon \le q\le x} \atop 
a(q) \equiv 0\!\!\! \pmod{p}}
\sideset{}{^*}{\sum}_{\substack{n\le x  \atop q||n}} 1.
\end{equation}
This is majorized by
\begin{equation}\label{october2}
\sideset{}{^*}{\sum}_{n\le x}
\sideset{}{^*}{\sum}_{\substack{x^\epsilon \le q\le x} \atop q||n} 1.
\end{equation}
The inner sum is bounded and so by \propref{nonzero},
we see that \eqref{october1} is
\begin{equation}\label{october3}
\ll x/\sqrt{\log x}.
\end{equation}
Now, consider the quantity
\begin{equation}\label{october4}
\sideset{}{^*}{\sum}_{\substack{x^{1/\log\log x} \le q\le 
x^\epsilon} \atop 
a(q) \equiv 0\!\!\! \pmod{p}}
\sideset{}{^*}{\sum}_{\substack{n\le x  \atop q||n}} 1.
\end{equation}
By \propref{nonzero}, the inner sum is
\begin{equation}\label{october5}
\ll\ x/q\sqrt{\log x}.
\end{equation}\label{october6}
Since
$$
\sum_{x^{1/\log\log x} \leq q \leq x^\epsilon} \frac{1}{q}
\ =\ \log\log\log x \ +\ O(1),
$$
it follows that \eqref{october4} is
\begin{equation}\label{october7}
\ll\ xL_3(x)/\sqrt{\log x}.
\end{equation}
Putting \eqref{october3} and \eqref{october7} together, we
deduce that 
\begin{eqnarray*}
\sideset{}{^*}{\sum}_{\substack{q\le x  \atop a(q) \equiv 0 \!\!\!\pmod{p}}} 
\sideset{}{^*}{\sum}_{\substack{n\le x  \atop q||n}} 1 
&=&
\sideset{}{^*}{\sum}_{\substack{ q \le x^{1/\log\log x}  
\atop a(q) \equiv 0\!\!\! \pmod{p}}}
\sideset{}{^*}{\sum}_{\substack{n\le x  \atop q||n}} 1 
\phantom{m} + O(xL_3(x)/\sqrt{\log x}). 
\end{eqnarray*}
Now by \propref{nonzero}, \lemref{obvious} (and the 
fact that in the sum $a^0(q)=1$), the sum on the right
 is equal to
\begin{eqnarray*}
&& (1 + o(1))\frac{u_f x}{\sqrt{\pi}} 
\sideset{}{^*}{\sum}_{\substack{q\le x^{1/\log\log x} \atop a(q) \equiv 0\!\!\! \pmod{p}}} 
\frac{1}{q\sqrt{\log x/q}}\left(1 + O\left(\frac{1}{q}\right)
+ O\left(\frac{1}{\log x/q}\right) \right) \\
&=& 
(1 + o(1))\frac{u_f x}{\sqrt{\pi}}
\sideset{}{^*}{\sum}_{\substack{q \le x^{1/\log\log x}  \atop a(q) 
\equiv 0 \!\!\! \pmod{p}}} \frac{1}{q\sqrt{\log x/q}} ~+~  
O\left(\frac{x}{(\log x)^{\frac{1}{2}}}\right). \\ 
\end{eqnarray*}
Now applying \lemref{1}, we see that this is
\begin{eqnarray*}
&=& (1 + o(1)){\frac{u_f \delta(p)x L_2(x)}{\sqrt{\pi}(\log x)^{\frac{1}{2}}}} 
~+~ O\left(\frac {x L_3(x)}{(\log x)^{\frac{1}{2}}}\right).
 \\
\end{eqnarray*}
This proves the lemma. 
\end{proof}

\begin{lem}\label{3}
Assume $p < y$. Then
\begin{eqnarray*}
\sideset{}{^*}{\sum}_{\substack{n\le x}} \nu(p,n)^2 = (1 + o(1))
\frac{u_f \delta^2(p) x L^2_2(x)}{\sqrt{\pi}(\log x)^{\frac{1}{2}}} + 
O\left( \frac{\delta(p) x L_2(x)L_3(x)}{(\log x)^{\frac{1}{2}}}\right).
\end{eqnarray*}
\end{lem}

\begin{proof}
The sum in question is equal to 
\begin{eqnarray*}
\sideset{}{^*}{\sum}_{\substack{q_1^{m_1} \le x  \atop a(q_1^{m_1}) 
\equiv 0\!\!\! \pmod{p}}} 
~~\sideset{}{^*}{\sum}_{\substack{q_2^{m_2}\le x  \atop a(q_2^{m_2}) 
\equiv 0 \!\!\! \pmod{p}}} 
~~\sideset{}{^*}{\sum}_{\substack{n\le x  
\atop q_1^{m_1}||n,~ q_2^{m_2}||n}} 1.
\end{eqnarray*}
By a small modification to the argument given in the proof of
\lemref{2}, we find that 
the contribution of terms with $q_1=q_2$ is 
$$
\ll\ \frac{xL_2(x)}{(\log x)^{1/2}}.
$$
Next, we consider the contribution $S$ (say) of 
terms with $q_1^{m_1}q_2^{m_2} > x^\epsilon$. 
For estimating this, we may suppose that $q_1^{m_1} > q_2^{m_2}$.
Since $q_2 \geq 2$, we may suppose that $x/2 \geq\ q_1^{m_1}
\geq \ x^{\epsilon/2}\ =\ z$ (say). 

Denote by $S_1$ the contribution of terms for which
$z \leq q_1^{m_1} \leq \sqrt{x/2}$ and by $S_2$ the contribution of
all remaining terms in $S$. Then by \propref{nonzero}, we have
\begin{eqnarray*}
S_1\ 
&\ll &  x \sideset{}{^*}{\sum}_{z \le \substack{q_1^{m_1} \le \sqrt{x/2}}}
\frac{1}{q_1^{m_1}}
\sum_{q_2^{m_2} \leq q_1^{m_1}}
\frac{1}{q_2^{m_2}\sqrt{\log \frac{x}{q_1^{m_1}q_2^{m_2}}}}
\\ \\
&\ll & x\sum_{z \leq q_1^{m_1} \leq \sqrt{x/2}}
\frac{1}{q_1^{m_1}\sqrt{\log \frac{x}{q_1^{2m_1}}}}
\log\log (q_1^{m_1}) \\ \\
&\ll & xL_2(x) \int_z^{\sqrt{x/2}} \frac{dt}{t(\log t)\sqrt{\log x/t^2}}
\ \ll\ \frac{xL_2(x)}{\sqrt{\log x}}.
\end{eqnarray*}
Next, we observe that
$$
S_2\ \ll\ \sum_{\sqrt{x/2}\ < q_1^{m_1} \leq x/2}~~~
\sideset{}{^*}{\sum}_{n \leq x/q_1^{m_1}} \nu(p,n)
$$
and by \lemref{2}, this is
$$
\ll \ xL_2(x)\sum_{\sqrt{x/2} < q_1^{m_1}\ \leq x/2}
\frac{1}{q_1^{m_1}}
\frac{1}{\sqrt{\log x/q_1^{m_1}}}
\ \ll\ \frac{xL_2(x)}{\sqrt{\log x}}.
$$
It remains to estimate
\begin{eqnarray*}
&& \sideset{}{^*}{\sum}_{\substack{q_1^{m_1}q_2^{m_2} \le x^{\epsilon}  
\atop {a(q_1^{m_1}) \equiv 0\!\!\! \pmod{p} 
\atop a(q_2^{m_2}) \equiv 0 \!\!\! \pmod{p}}}} 
~~\sideset{}{^*}{\sum}_{\substack{n\le x \atop q_1^{m_1}||n,~ q_2^{m_2}||n }} 1 \\ 
&& \phantom{mm} =~~ 
I + J, \phantom{m} {\rm say}
\end{eqnarray*}
where in $I$ we have the terms with $m_1 > 1$ or $m_2 > 1$ and in $J$ 
we have the terms with $m_1 = m_2 = 1$. 
In order to estimate $I$, suppose without loss of generality that
$m_1 \ge 2$. Then by \propref{nonzero}, we have
\begin{eqnarray*}
I & \ll & 
x \sideset{}{^*}{\sum}_{\substack{q_1^{m_1} \atop m_1\ge 2 }} \frac{1}{q_1^{m_1}}
~~\sideset{}{^*}{\sum}_{\substack{q_2^{m_2} \atop q_1^{m_1}q_2^{m_2} \le 
x^{\epsilon}}} \frac{1}{q_2^{m_2}\sqrt{\log \frac{x}{q_1^{m_1}q_2^{m_2}}}} \\ \\
& \ll & 
\frac{x}{\sqrt{\log x}} \sideset{}{^*}{\sum}_{\substack{q_1^{m_1} 
\atop m_1\ge 2 }} \frac{1}{q_1^{m_1}} 
\left(\sum _{q_2 \le x^{\epsilon}} \frac{1}{q_2} ~+~ \sum_{q_2 \atop m_2 \ge 2} 
\frac{1}{q_2^{m_2}}\right) \\ \\
& \ll &  \frac{x L_2(x)}{\sqrt{\log x}}.  
\end{eqnarray*}
Next, we consider 
\begin{eqnarray*}
J & = & \sideset{}{^*}{\sum}_{\substack{q_1q_2 \le x^{\epsilon}  
\atop {a(q_1) \equiv 0\!\!\! \pmod{p} 
\atop a(q_2) \equiv 0 \!\!\! \pmod{p}}}} 
~~\sideset{}{^*}{\sum}_{\substack{n\le x \atop q_1||n,~ q_2||n }} 1 \\ 
\end{eqnarray*}
By \propref{nonzero} and \propref{xi}, we have
\begin{eqnarray*}
J &=&  (1 + o(1))
\frac{u_f x}{\sqrt{\pi\log x}}
\sideset{}{^*}{\sum}_{\substack{q_1q_2 \le x^{\epsilon} 
\atop {a(q_1) \equiv 0\!\!\! \pmod{p} \atop {a(q_2)\equiv 0\!\!\! \pmod{p} \atop q_1 \ne q_2}}}} 
\frac{1}{q_1q_2} ~+~  O\left(\frac{xL_2(x)}{\sqrt{\log x}}\right)  \\ \\
&=& (1 + o(1))
\frac{u_f x}{\sqrt{\pi\log x}} 
\left(\sideset{}{^*}{\sum}_{\substack
{q \le x  \atop a(q) \equiv 0 \!\!\! \pmod{p}}} \frac{1}{q} 
\right)^2 ~+~  O\left(\frac{xL_2(x)}{\sqrt{\log x}}\right) \\ \\
&=& (1 + o(1))
\frac{u_f x}{\sqrt{\pi\log x}} \left(\delta(p)L_2(x) 
+ O(L_3(x)) \right)^{2} ~+~ O\left(\frac{x L_2(x)}{\sqrt{\log x}} \right)  
\end{eqnarray*}
\begin{eqnarray*}
&=& (1 + o(1))
\frac{u_f \delta^2(p)x L_2^2(x)}{\sqrt{\pi}(\log x)^{\frac{1}{2}}} ~+~ 
O\left(\delta(p)\frac{xL_2(x)L_3(x)}{\sqrt{\log x}}\right).
\end{eqnarray*}
This proves the lemma.
\end{proof}

\begin{lem}\label{4}
Suppose $p< y$, then 
\begin{eqnarray*}
\sideset{}{^*}{\sum}_{\substack{n\le x}}\left(\nu(p,n) 
- \delta(p) L_2(x)\right)^2 
\ll \frac{\delta(p) x}{(\log x)^{\frac{1}{2}}} L_2(x)L_3(x).
\end{eqnarray*}
\end{lem}

\begin{proof}
This follows from \lemref{2} and \lemref{3}.
\end{proof}

\begin{lem}\label{5}
Assume $p < y$, then
\begin{eqnarray*}
\# \left\{ n\le x \mid \nu(p,n) = 0 \right \} 
\ll \frac{xL_3(x)}{\delta(p) (\log x)^{\frac{1}{2}} L_2(x)}.
\end{eqnarray*}
\end{lem}

\begin{proof}
By \lemref{4}, this is
\begin{eqnarray*}
\ll \frac{1}{\delta^2(p)L_2^2(x)} \left\{\delta(p)\frac{x}
{(\log x)^{\frac{1}{2}}}L_2(x)L_3(x)\right\}
 = \frac{xL_3(x)}{\delta(p)(\log x)^{\frac{1}{2}}L_2(x)}.
\end{eqnarray*}
\end{proof}

\section{\bf Proof of \thmref{thm1}}\label{proofofmainresult} 

For a prime $p$, let 
$$
G_p(x)= \# \left\{ n \le x \mid p|n, (n, a(n))=1, q|n \Rightarrow q\ge p \right\}
$$ 
and  
$G(x) = \displaystyle\sum_{\substack{p\le x}} G_p(x) = A_1 + A_2+ A_3$, 
where 
\begin{eqnarray*}
A_1 &=& \displaystyle \sum_{p \le L_2^{\frac{1}{2} -\epsilon}(x)} G_p (x),\\
A_2 &=& \displaystyle \sum_{L_2^{\frac{1}{2} - \epsilon}(x) < p 
< L_2^{1 +\epsilon}(x)} G_p(x), \\
A_3 &=&  \displaystyle \sum_{p \ge L_2^{1 + \epsilon}(x)} G_p(x).
\end{eqnarray*}
Now, using \lemref{5}, we have 
\begin{eqnarray}
A_1 & \le & \displaystyle \sum_{p \le L_2^{\frac{1}{2} -\epsilon}(x)}
\#\left\{n \le x \mid p|n, (n, a(n))=1\right\} 
\end{eqnarray}

\begin{eqnarray}
& \ll & \frac{x L_3(x)}{(\log x)^{\frac{1}{2}} L_2(x)} 
\displaystyle \sum_{p \le L_2^{\frac{1}{2} -\epsilon}(x)}
\frac{1}{\delta(p)} \nonumber~  \\
& \ll & \frac{x L_3(x)}{(\log x)^{\frac{1}{2}} L_2(x)} 
\displaystyle \sum_{1 \ll p \le L_2^{\frac{1}{2} -\epsilon}(x)}
 p \nonumber~, \phantom{m} {\rm as}~~\delta(p) \gg \frac{1}{p} \\
& \ll & \frac{x}{(\log x)^{\frac{1}{2}} L_2^{\epsilon}(x)} 
~=~ o\left( \frac{x}{\left(L_3(x)\log x \right)^{\frac{1}{2}}} \right).
\end{eqnarray}
Moreover, by \lemref{sieve2}, we have
\begin{eqnarray*}
A_2 & \le & \displaystyle 
\sum_{L_2^{\frac{1}{2} - \epsilon}(x) < p < L_2^{1 +\epsilon}(x)} 
\#\left\{n \le x \mid p|n, a(n)\ne 0, q|n \Rightarrow q\ge p\right\} \\
& \ll & \frac{x}{(\log x)^{\frac{1}{2}}} 
\displaystyle \sum_{L_2^{\frac{1}{2} - \epsilon}(x) < p < L_2^{1 +\epsilon}(x)} 
\frac{1}{p}\prod_{l\le p \atop l~{\rm prime}}\left(1 - \frac{1}{l}\right) \\
& \ll & \frac{x}{(\log x)^{\frac{1}{2}}} 
\displaystyle \sum_{L_2^{\frac{1}{2} - \epsilon}(x) < p < L_2^{1 +\epsilon}(x)} 
\frac{1}{p \log p}\\
& \ll & \frac{x}{L_3(x) (\log x)^{\frac{1}{2}}}
~~~~~~~~=~~~~o\left( \frac{x}{\left(L_3(x) \log x \right)^{\frac{1}{2}}} \right)  .
\end{eqnarray*}
Let $y_1 = L_2(x)^{1 + \epsilon}$ and as in \eqref{Ny1},
$N_{y_1}(x) = \# \left\{n\le x \mid q|n \Rightarrow q \ge y_1, 
~ a(n) \ne 0 \right\}$. 
Then 
\begin{eqnarray*}
N_{y_1}(x) ~- \sideset{}{^*}{\sum}_{\substack{y_1 \le q_1^{m}, ~q_2 \le x  \atop a(q_1^{m}) 
\equiv 0 \!\!\! \pmod{q_2}}}~~
\sideset{}{^{**}}{\sum}_{\substack{n\le x \atop {q_1^{m}||n, ~q_2|n}}}  1 
~~\le A_3 ~~\le  N_{y_1}(x),
\end{eqnarray*}
where $\sideset{}{^{**}}{\sum}_{\substack{n\le x}}$ means that 
the summation is over all natural numbers $n\le x$ 
such that $a(n) \ne 0$ and $q|n$ implies that $q>y_1$.\\

\noindent
By \lemref{sieve1}, to prove the theorem, it suffices to show that 
\begin{eqnarray}\label{11}
\sideset{}{^*}{\sum}_{\substack{y_1 \le q_1^{m}, ~q_2 \le x 
\atop a(q_1^{m}) \equiv 0 \!\!\! \pmod{q_2}}}
~~\sideset{}{^{**}}{\sum}_{\substack{n\le x \atop q_1^{m}||n, ~q_2|n}} 1 
~=~ o\left(\frac{x}{\left(L_3(x)\log x \right)^{\frac{1}{2}}} \right).
\end{eqnarray} 
In order to prove \eqref{11}, let us write
\begin{eqnarray*}
\sideset{}{^*}{\sum}_{\substack{y_1 \le q_1^{m}, ~q_2 \le x \atop a(q_1^m) 
\equiv 0 \!\!\! \pmod{q_2}}} 
~~~~~\sideset{}{^{**}}{\sum}_{\substack{n\le x  \atop q_1^{m}||n, ~q_2|n}} 1
& = & \sideset{}{^*}{\sum}_{\substack{y_1 \le q_1^m , ~q_2 \le x, ~m\ge 2 
\atop a(q_1^{m}) \equiv 0 \!\!\! \pmod{q_2}}} 
~~~\sideset{}{^{**}}{\sum}_{\substack{n\le x, 
\atop {q_1^m||n,~ q_2|n}}} 1 \\ \\
& + &  \sideset{}{^*}{\sum}_{\substack{y_1 \le q_1,~q_2 \le x
\atop a(q_1) \equiv 0 \!\!\! \pmod{q_2}}} 
~~~~~\sideset{}{^{**}}{\sum}_{\substack{n\le x \atop q_1||n, ~q_2|n}} 1 \\ \\
&= & B_1 + B_2~.
\end{eqnarray*}
Let us consider $B_1$ first. The terms for which $q_1^mq_2 \ge (\log x) x^{1/2} y_1^2$
contribute an amount which is
\begin{eqnarray*}
&\ll & \frac{\sqrt{x}}{\log x} 
\sum_{q_2 \le x} \frac{1}{q_2}
\sum_{q_1^m \ge y_1 \atop m\ge2} \frac{1}{q_1^m}\\
&\ll & \frac{\sqrt{x}}{y_1 \log x}L_2(x) ~\ll~ \frac{x}{L_2^{\epsilon}(x) \log x}.
\end{eqnarray*}
For the remaining terms $q_1^mq_2 \le (\log x)x^{1/2}y_1^2$.
We use \propref{nonzero} to see that the remaining terms in $B_1$ are
\begin{eqnarray*}
&\ll & \frac{x}{(\log x)^{\frac{1}{2}}} 
\displaystyle \sum_{y_1 \le q_2 \le x} \frac{1}{q_2} 
~\displaystyle \sum_{y_1 \le q_1^m  \atop m \ge 2} \frac{1}{q_1^m} \\
& \ll & \frac{x}{y_1(\log x)^{\frac{1}{2}}} ~\displaystyle \sum_{y_1 \le q_2 \le x}
\frac{1}{q_2} \\
& \ll & \frac{xL_2(x)}{y_1 (\log x)^{\frac{1}{2}}} 
~=~ \frac{x}{(\log x)^{\frac{1}{2}}L_2^{\epsilon}(x)} .
\end{eqnarray*}
For $B_2$, we observe that if $a(q_1) \ne 0$ and $a(q_1) \equiv 0 \pmod{q_2}$,
then $q_2 \le |a(q_1)| \le 2 \sqrt{q_1}$. Hence $q_1 \ge q_2^2/4$ and so 
$q_1q_2 \ge q_2^3/4$. Hence for the inner sum in $B_2$ to be nonempty, we need 
$q_2 \le (4x)^{1/3}$. Thus  
\begin{eqnarray*}
B_2 &=& \sideset{}{^*}{\sum}_{\substack{y_1 \le q_1 \le x 
\atop {y_1 \le q_2 \le (4x)^{1/3}
\atop a(q_1) \equiv 0 \!\!\! \pmod{q_2}}}}
~~~~~\sideset{}{^{**}}{\sum}_{\substack{n\le x \atop q_1||n, ~q_2|n}} 1 \\
&=&\sideset{}{^*}{\sum}_{\substack{y_1 \le q_1 \le \sqrt{x} 
\atop {y_1 \le q_2 \le 2\sqrt{q_1}
\atop a(q_1) \equiv 0 \!\!\! \pmod{q_2}}}}
~~~~~\sideset{}{^{**}}{\sum}_{\substack{n\le x \atop q_1||n, ~q_2|n}} 1
~+~ \sideset{}{^*}{\sum}_{\substack{\sqrt{x} \le q_1 \le x 
\atop {y_1 \le q_2 \le 2\sqrt{q_1}
\atop a(q_1) \equiv 0 \!\!\! \pmod{q_2}}}}
~~~~~\sideset{}{^{**}}{\sum}_{\substack{n\le x \atop q_1||n, ~q_2|n}} 1\\\\
&=& D_1 + D_2 ~.
\end{eqnarray*}
Then by \propref{nonzero} and the fact that $q_1q_2 \ll x^{3/4}$,
we have
\begin{eqnarray*}
D_1 &\ll &  \frac{x}{(\log x)^{\frac{1}{2}}}
\sideset{}{^*}{\sum}_{\substack{y_1 \le q_1 \le \sqrt{x} \atop
{y_1 \le q_2 \le 2\sqrt{q_1} \atop a(q_1) \equiv 0 \!\!\! \pmod{q_2}}}} \frac{1}{q_1q_2} \\ \\
& = & \frac{x}{(\log x)^{\frac{1}{2}}} \left\{
\sideset{}{^*}{\sum}_{\substack{y_1 \le q_2 \le 2 x^{1/4}
\atop{~\frac{1}{4}q_2^2 \le q_1 \le q_2^2 \log q_2 \atop a(q_1) 
\equiv 0 \!\!\! \pmod{q_2}}}}
\frac{1}{q_1q_2} ~+~ \sideset{}{^*}{\sum}_{\substack
{y_1 \le q_2 \le 2 x^{1/4}  \atop{q_2^2 \log q_2 \le q_1 \le \sqrt{x}
\atop a(q_1) \equiv 0 \!\!\! \pmod{q_2}}}} \frac{1}{q_1q_2}
\right\}.
\end{eqnarray*}
By \propref{prop100}, the second sum is
\begin{eqnarray*}
\ll  \frac{x L_2(x)}{(\log x)^{\frac{1}{2}}} 
\displaystyle \sum_{y_1 \le q_2 \le 2 x^{1/4}}\frac{1}{q_2^2}
\ll  \frac{x L_2(x)}{y_1(\log x)^{\frac{1}{2}}} 
~=~  \frac{x}{(\log x)^{\frac{1}{2}}L_2^{\epsilon}(x)}.
\end{eqnarray*}
The first sum is
\begin{eqnarray*}
&\ll & \frac{x}{(\log x)^{\frac{1}{2}}}
\displaystyle \sum_{\frac{1}{4}y_1^2 \le q_1 \le x} ~\frac{1}{q_1}
 ~~~\sideset{}{^*}{\sum}_{\substack{\sqrt{\frac{q_1}{\log q_1}} \le q_2 \le 2\sqrt{q_1} 
\atop a(q_1) \equiv 0 \!\!\! \pmod{q_2}}} \frac{1}{q_2} .
\end{eqnarray*}
We note that the inner sum over $q_2$ is bounded. In fact with 
$0<|a(q_1)|\le 2\sqrt{q_1}$, there exists at most
one $q_2 \ge \sqrt{\frac{q_1}{\log q_1}}$ which divides $a(q_1)$. 
Thus, the right hand side is
\begin{eqnarray*}
\ll  \frac{x}{(\log x)^{\frac{1}{2}}}\displaystyle \sum_{y_1 \le q_1 \le x} 
~\frac{\sqrt{\log q_1}}{q_1^{3/2}} \ll \frac{x}{(L_2(x) \log x)^{\frac{1}{2}}}.
\end{eqnarray*}
In order to estimate $D_2$, we write
\begin{eqnarray*}
D_2 &=& 
\sideset{}{^*}{\sum}_{\substack{y_1 \le q_2 \le e^{\sqrt{\log x}}\\
\atop {\sqrt{x} \le q_1 \le \frac{x}{2q_2} 
\atop a(q_1) \equiv 0 \!\!\! \pmod{q_2}}}}
~~~~~\sideset{}{^{**}}{\sum}_{\substack{n\le x \atop q_1||n, ~q_2|n}} 1
\qquad~~~~~~~~~~~~~~~+~~~~~~~~~~
\sideset{}{^*}{\sum}_{\substack{y_1 \le q_2 \le e^{\sqrt{\log x}}
\atop {\frac{x}{2q_2} \le q_1 \le  \frac{x}{q_2}
\atop a(q_1) \equiv 0 \!\!\! \pmod{q_2}}}}
~~~~~\sideset{}{^{**}}{\sum}_{\substack{n\le x \atop q_1||n, 
~q_2|n}} 1 \\ \\
&& ~~~~
+ \sideset{}{^*}{\sum}_{\substack{e^{\sqrt{\log x}}\le q_2 \le
\left(\frac{x}{\log x}\right)^{1/3}\\
 \atop {\sqrt{x} \le q_1 \le \frac{x}{q_2}
\atop a(q_1) \equiv 0 \!\!\! \pmod{q_2}}}}
~~~~~\sideset{}{^{**}}{\sum}_{\substack{n\le x \atop q_1||n, ~q_2|n}} 1 
~~~~~~~~~~~~+~~~~~~~~~~~~
 \sideset{}{^*}{\sum}_{\substack{\left(\frac{x}{\log x}\right)^{1/3} 
\le q_2 \le x \atop {\sqrt{x} \le q_1 \le \frac{x}{q_2}
\atop a(q_1) \equiv 0 \!\!\! \pmod{q_2}}}}
~~~~~\sideset{}{^{**}}{\sum}_{\substack{n\le x \atop q_1||n, ~q_2|n}} 1\\ \\
&=& J_1 + J_2 + J_3 + J_4 ~.
\end{eqnarray*}
Here
\begin{eqnarray*}
J_4 & \ll &
x \sideset{}{^*}{\sum}_{\substack{\sqrt{x} \le q_1 \le x}} \frac{1}{q_1}
\sideset{}{^*}{\sum}_{\substack{\left(\frac{x}{\log x}\right)^{1/3}
\le q_2 \le 2^{2/3}x^{1/3}}} \frac{1}{q_2} \\ \\
& \ll & x^{2/3}(\log x)^{1/3} \pi((4x)^{1/3})
\sideset{}{^*}{\sum}_{\substack{\sqrt{x} \le q_1 \le x}} \frac{1}{q_1}\\
\end{eqnarray*}
where $\pi(t)$ denotes the number of primes $\le t$. Thus
\begin{eqnarray*}
J_4 & \ll & \frac{x}{(\log x)^{2/3}} 
\sideset{}{^*}{\sum}_{\substack{\sqrt{x} \le q_1 \le x}} \frac{1}{q_1} 
~\ll~  \frac{xL_2(x)}{(\log x)^{2/3}}
\end{eqnarray*}
and
\begin{eqnarray*}
J_3  &\ll & 
x \sideset{}{^*}{\sum}_{\substack {\sqrt{x} \le q_1 \le x}} \frac{1}{q_1} 
\sideset{}{^*}{\sum}_{\substack{q_2|a(q_1) \atop q_2 \ge e^{\sqrt{\log x}}}}\frac{1}{q_2} \\ \\
&\ll & 
\frac{x}{e^{\sqrt{\log x}}} 
\sideset{}{^*}{\sum}_{\substack {\sqrt{x} \le q_1 \le x}} \frac{1}{q_1} 
\# \left\{q_2 \mid q_2 \ge e^{\sqrt{\log x}}, q_2|a(q_1), 0 \ne 
a(q_1) \le 2\sqrt{x} \right\} \\
&\ll &
\frac{x \sqrt{\log x}}{e^{\sqrt{\log x}}}
\sum_{\substack {q_1 \le x}} \frac{1}{q_1} 
\ll \frac{x \sqrt{\log x}~L_2(x)}{e^{\sqrt{\log x}}}.
\end{eqnarray*}
In order to estimate $J_1$ and $J_2$, we write
\begin{eqnarray*}
J_1 &=& \sideset{}{^*}{\sum}_{\substack{y_1 \le q_2 \le e^{\sqrt{\log x}}\\
\atop {\sqrt{x} \le q_1 \le \frac{x}{2q_2} 
\atop a(q_1) \equiv 0 \!\!\! \pmod{q_2}}}}
~~~~~\sideset{}{^{**}}{\sum}_{\substack{n\le x \atop q_1||n, ~q_2||n}} 1
\phantom{m}+ \phantom{m}
\sideset{}{^*}{\sum}_{\substack{y_1 \le q_2 \le e^{\sqrt{\log x}}\\
\atop {\sqrt{x} \le q_1 \le \frac{x}{2q_2} 
\atop a(q_1) \equiv 0 \!\!\! \pmod{q_2}}}}
~~~~~\sideset{}{^{**}}{\sum}_{\substack{n\le x \atop {q_1||n
\atop q_2^m||n, m\ge 2}}} 1 \\ \\
&=& J_{11} + J_{12}
\end{eqnarray*}
and
\begin{eqnarray*}
J_2 &=&
\sideset{}{^*}{\sum}_{\substack{y_1 \le q_2 \le e^{\sqrt{\log x}}
\atop {\frac{x}{2q_2} \le q_1 \le  \frac{x}{q_2}
\atop a(q_1) \equiv 0 \!\!\! \pmod{q_2}}}}
~~~~~\sideset{}{^{**}}{\sum}_{\substack{n\le x \atop q_1||n, 
~q_2||n}} 1
\phantom{m} + \phantom{m} 
\sideset{}{^*}{\sum}_{\substack{y_1 \le q_2 \le e^{\sqrt{\log x}}
\atop {\frac{x}{2q_2} \le q_1 \le  \frac{x}{q_2}
\atop a(q_1) \equiv 0 \!\!\! \pmod{q_2}}}}
~~~~~\sideset{}{^{**}}{\sum}_{\substack{n\le x \atop {q_1||n 
\atop q_2^m||n, m \ge 2}}} 1 \\ \\
&=& J_{21} + J_{22} ~~~~.
\end{eqnarray*}
We show that $J_{11}$ and $J_{21}$ are $o\left(x/L_3(x)(\log x)^{1/2}\right)$. 
Similarly, one can show that
$J_{12}$ and $J_{22}$ are $o\left(x/L_3(x)(\log x)^{1/2}\right)$.
We can write
\begin{eqnarray*}
J_{11} & \ll &
x \sideset{}{^*}{\sum}_{\substack{y_1 \le q_2 
\le e^{\sqrt{\log x}}}} \frac{1}{q_2} 
\sideset{}{^*}{\sum}_{\substack{\sqrt{x} \le q_1 \le \frac{x}{2q_2}
\atop a(q_1) \equiv 0 \pmod{q_2}}} 
\frac{1}{q_1 \left(\log \frac{x}{q_1q_2}\right)^{1/2}} \\ \\
& \ll &
x \sideset{}{^*}{\sum}_{\substack{y_1 \le q_2 
\le e^{\sqrt{\log x}}}} \frac{1}{q_2} 
~\int_{\sqrt{x}}^{x/2q_2} \frac{d\pi^{*}(t, q_2)}
{t \left(\log \frac{x}{q_2t}\right)^{1/2}}\\ \\ 
& \ll &
x \sideset{}{^*}{\sum}_{\substack{y_1 \le q_2 
\le e^{\sqrt{\log x}}}} \frac{1}{q_2} \left[\left\{\frac{\pi^{*}(t, q_2)}
{t \left(\log \frac{x}{q_2t}\right)^{1/2}}\right\}_{t = \sqrt{x}}^{t =x/2q_2} 
+ \int_{\sqrt{x}}^{x/2q_2} \frac{\pi^{*}(t,q_2)~dt}{t^2\left(\log 
\frac{x}{q_2t}\right)^{1/2} }\right]. 
\end{eqnarray*}
Then by using \thmref{thm 5.1}, we have
\begin{eqnarray*}
J_{11} &\ll & 
x \sideset{}{^*}{\sum}_{\substack{y_1 \le q_2 
\le e^{\sqrt{\log x}}}} \frac{1}{q_2^2} \left[\left\{\frac{1}{\log t 
\left(\log \frac{x}{q_2t}\right)^{1/2}}\right\}_{t = \sqrt{x}}^{t =x/2q_2} 
+ \int_{\sqrt{x}}^{x/2q_2} \frac{dt}{t \log t 
\left(\log \frac{x}{q_2t}\right)^{1/2} }\right] \\ \\
&\ll &
\frac{x}{\log x} \sideset{}{^*}{\sum}_{\substack{y_1 \le q_2 
\le e^{\sqrt{\log x}}}} \frac{1}{q_2^2} \left[1
+ \int_{\sqrt{x}}^{x/2q_2} \frac{dt}{t\left(\log \frac{x}{q_2t}\right)^{1/2} }\right] \\ \\
&\ll &
\frac{x}{(\log x)^{1/2}} \sideset{}{^*}{\sum}_{\substack{y_1 \le q_2 
\le e^{\sqrt{\log x}}}} \frac{1}{q_2^2} ~\ll~
\frac{x}{y_1(\log x)^{1/2}}.
\end{eqnarray*}
Since for each pair of primes $q_1, q_2$ with $y_1 \le q_2 \le e^{\sqrt{\log x}}, 
x/2q_2 \le q_1 \le x/q_2$, there
are at most two $n\le x$ with $q_1q_2\mid n$, we have
\begin{eqnarray*}
J_{21} &\ll &
\sideset{}{^*}{\sum}_{\substack{y_1 \le q_2 
\le e^{\sqrt{\log x}}}}\sideset{}{^*}{\sum}_{\substack{\frac{x}{2q_2} \le q_1 \le  
\frac{x}{q_2} \atop a(q_1) \equiv 0 \!\!\! \pmod{q_2}}}1 \\ \\ 
&\ll & 
\sideset{}{^*}{\sum}_{\substack{y_1 \le q_2 
\le e^{\sqrt{\log x}}}} \pi^{*}(x/q_2,~ q_2) ~\ll~ 
\frac{x}{\log x} \sideset{}{^*}{\sum}_{\substack{y_1 \le q_2 
\le e^{\sqrt{\log x}}}} \frac{1}{q_2^2},
 ~~~{\rm~ by~\thmref{thm 5.1}} \\ \\
&\ll &
\frac{x}{y_1\log x}.
\end{eqnarray*}
Hence
\begin{eqnarray*}
B_1 + B_2 ~=~  o\left(\frac{x}{(\log x)^{\frac{1}{2}} L_3(x)} \right).
\end{eqnarray*}
This completes the proof.

\smallskip
\noindent
{\bf Acknowledgments:} This work started while the first 
author was a post-doctoral fellow at the 
University of Toronto--Mississauga(UTM). 
It is her pleasure to thank UTM for financial support and 
excellent working conditions. Both authors would like to thank 
Ram Murty and the
referee for very helpful comments which enabled them to eliminate 
some errors from an earlier version.

\end{document}